\documentclass[11pt]{amsart}

\usepackage{amsmath}
\usepackage{amstext}
\usepackage{amssymb}
\usepackage{amsthm}
\usepackage{enumerate}
\usepackage{bbm} 

\makeatletter
\def\imod#1{\allowbreak\mkern10mu({\operator@font mod}\,\,#1)}
\makeatother

\newtheorem{theorem}{Theorem}[section]

\newtheorem{lemma}[theorem]{Lemma}
\newtheorem{claim}[theorem]{Claim}

\newtheorem*{clm}{Claim}

\theoremstyle{definition}
\newtheorem{definition}[theorem]{Definition}

\newtheorem{corollary}[theorem]{Corollary}
\theoremstyle{remark}
\newtheorem{remark}[theorem]{Remark}

\theoremstyle{remark}

\numberwithin{equation}{section}

    \DeclareMathOperator{\dom}{dom}

        \DeclareMathOperator{\WO}{WO}

    \newcommand{\forces}{\Vdash}

    \newcommand{\card}[1]{\vert #1\vert}
    \newcommand{\seq}[1]{\langle #1 \rangle }

\def\Q{{\mathbb Q}}

\def\BbM{{\mathbb M}}

\begin{document}

\title[Definable Towers]{Definable Towers}



\author{Vera Fischer}

\address{Kurt G\"odel Research Center, University of Vienna, W\"ahringer Strasse 25, 1090 Vienna, Austria}
\email{vfischer@logic.univie.ac.at}

\author{Jonathan Schilhan}

\address{Kurt G\"odel Research Center, University of Vienna, W\"ahringer Strasse 25, 1090 Vienna, Austria}
\email{jonathan.schilhan@univie.ac.at}

\thanks{\emph{Acknowledgments:} The authors would like to thank the Austrian Science Fund, FWF, for generous support through START-Project Y1012-N35.}





\begin{abstract} 
We study the definability of maximal towers and of inextendible linearly ordered towers (ilt's), a notion that is more general than that of a maximal tower. We show that there is, in the constructible universe, a $\Pi^1_1$ definable maximal tower that is indestructible by any proper Suslin poset. We prove that the existence of a $\Sigma^1_2$ ilt implies that $\omega_1 = \omega_1^L$. Moreover we show that analogous results hold for other combinatorial families of reals. We prove that there is no ilt in Solovay's model. And finally we show that the existence of a $\Sigma^1_2$ ilt is equivalent to that of a $\Pi^1_1$ maximal tower.  
\end{abstract}

\maketitle

\section{Introduction}

The definability of various combinatorial families of reals has become a very active research area in set theory over the last decades. Among these families we find maximal almost disjoint families (\cite{BK1},\cite{FFK1},\cite{ADM1},\cite{RS1},\cite{AT1},\cite{AT2}), maximal cofinitary groups (\cite{FST1},\cite{HS1},\cite{BAK1}), maximal eventually different families (\cite{FS1},\cite{DS1}), maximal families of orthogonal measures (\cite{FT1}) or maximal independent families (\cite{BFK1}), just to name a few. One of the cornerstones of the theory has been laid by A. Miller in his seminal paper \cite{AM1} in which he shows, among other things, how to construct coanalytic families of various such kinds in the constructible universe. We continue along these lines by studying the definablity of \emph{maximal towers} and \emph{inextendible linearly ordered towers} (abbreviated as \emph{ilt}). 

A tower will be, as usual, a set $X \subseteq [\omega]^\omega$ which is well ordered with respect to reverse almost inclusion, i.e. the relation $x \leq y$ given by $\exists n \in \omega (y \setminus n \subseteq x )$. A tower is maximal if it has no pseudointersection. In the definition of a linearly ordered tower we drop the requirement that the order is well-founded. An inextendible linearly ordered tower is one that has no top-extension, i.e. has no pseudointersection.

The questions that we will ask and answer for towers are inspired to a great extent by those that appeared in relation to mad families. Recall that two sets $x, y \in [\omega]^\omega$ are called almost disjoint whenever $x \cap y$ is finite. An almost disjoint family is a subset of $[\omega]^\omega$ all of whose elements are pairwise almost disjoint. A maximal almost disjoint family (mad family) is an infinite almost disjoint family that cannot be properly extended to a larger one. For mad families, the story begins with Mathias' influential work \cite{ADM1} in which he showed that mad families cannot be analytic. 

In Section 2 we will show that neither maximal towers nor ilt's can be analytic (Theorem~\ref{thm:nosigma11towers} and Theorem~\ref{thm:noanalyticmlt}). On the other hand we prove in Section 3, as a main result, that $\Pi^1_1$ maximal towers do exist in $L$ (Theorem~\ref{thm:coanalytictower}), using the technique developped by Miller in \cite{AM1}.

Another topic that has been studied extensively for mad families is the existence of $\Pi^1_1$ examples in various forcing extensions. For instance it has been shown in \cite{BK1} that there is a $\Pi^1_1$ mad family in a model obtained by adding Hechler reals. In Section 4 we will outshadow all these questions for towers by showing that in $L$ there is a $\Pi^1_1$ maximal tower that is indestructible by any proper Suslin partial order (Theorem~\ref{thm:indestructibletower}). 

Section 5 deals with the value of $\omega_1$ in models where ilt's can have simple definitions. As a main result we show that the existence of a $\Sigma^1_2(x)$ ilt implies that $\omega_1 = \omega_1^{L[x]}$ (Theorem~\ref{thm:omega1}). The same has been shown for mad families in \cite{AT1}. Using similar ideas we show that this holds analogously for maximal independent families, Hamel bases and ultrafilters (Theorem~\ref{thm:independent},~\ref{thm:hamel} and ~\ref{thm:ultrafilter}). In \cite{BK1} Brendle and Khomskii ask whether there is some notion of transcendence over $L$ that is equivalent to the non-existence of a $\Pi^1_1$ mad family. The same question can be asked for other families and our observations contribute to this question by giving a sufficient condition of this kind.   

In Section 6 we show that there is no ilt in Solovay's model (Theorem~\ref{thm:solovay}). For mad families this was a long standing open question first asked by Mathias in \cite{ADM1} and solved by T\"ornquist in \cite{AT1}.

In Section 7 we show that the existence of a $\Sigma^1_2$ ilt is equivalent to that of a $\Pi^1_1$ ilt which is equivalent to that of a $\Pi^1_1$ tower (Theorem~\ref{thm:bigequivalence}). This theorem fits into a series of results stating that we can canonicaly construct $\Pi^1_1$ objects from given $\Sigma^1_2$ ones. For mad families this was shown in \cite{AT2}. For maximal independent families see \cite{BFK1} and for maximal eventually different families see \cite{FS1}.  

We will always stress the difference between lightface ($\Pi^1_1, \Sigma^1_1, \Sigma^1_2$) and boldface ($\mathbf{\Pi}^1_1,\mathbf{\Sigma}^1_1, \mathbf{\Pi}^1_2$) definitions as well as definitions relative to a fixed real parameter ($\Pi^1_1(x), \Sigma^1_1(x), \Sigma^1_2(x)$) to stay as general as possible.

\section{Towers and Definability}

\begin{definition}\label{max_tower}
 A tower is a set $X \subseteq [\omega]^\omega$ which is well ordered with respect to the relation defined by $x \leq y$ iff $y \subseteq^* x$. It is called maximal if it cannot be further extended, i.e. it has no pseudointersection.
\end{definition}

\begin{theorem}
\label{thm:nosigma11towers}
A tower contains no (uncountable) perfect set, i.e. is thin. In particular there is no $\mathbf{\Sigma}^1_1$ maximal tower. 
\end{theorem}

\begin{proof}
 Assume $X \subseteq [\omega]^{\omega}$ is a tower and $P \subseteq X$ is a perfect set. The set $R = \{ (x,y) : x,y \in P \wedge y \subseteq^* x \}$ is Borel. $P$ is an uncountable Polish space and $R$ is Borel as a subset of $P \times P$ . But $R$ is a well order of $P$, which contradicts $R$ having the Baire property by \cite[Theorem 8.48]{AK}. A maximal tower must be uncountable and an uncountable analytic set has a perfect subset by the Perfect Set Theorem. Thus there is no analytic maximal tower.  
\end{proof}

\begin{theorem}\label{thm:towersubsetofL}
 Every ${\Sigma}^1_2(x)$ tower is a subset of $L[x]$ and thus of size at most $\omega_1^{L[x]}$.
\end{theorem}

\begin{proof}
 If $X$ is a ${\Sigma}^1_2(x)$ tower then it contains no perfect set and is thus a subset of $L[x]$ by the Mansfield-Solovay Theorem \cite[Theorem 21.1]{AM}. 
\end{proof}

\begin{corollary}
 The existence of a $\Sigma^1_2(x)$ maximal tower implies that $\omega_1 = \omega_1^{L[x]}$. 
\end{corollary}

All of the proofs above rely mostly on the fact that towers exhibit a well ordered structure and the maximality is inessential. Thus it is natural to ask for a more general version of a tower which is not trivially ruled out by an analytic definition. We call a set $X \subseteq [\omega]^{\omega}$ an inextendible linearly ordered tower (abbreviated as ilt) if it is linearly ordered with respect to $\subseteq^*$ and has no pseudointersection. We call $Y \subseteq X$ cofinal in $X$ if $\forall x \in X \exists y \in Y (y \subseteq^* x)$.

\begin{theorem}
\label{thm:noanalyticmlt}
 There is no $\mathbf{\Sigma}^1_1$ definable inextendible linearly ordered tower.
\end{theorem}

\begin{proof}
 Assume $X = p[T]$ is an ilt where $T$ is a tree on $2 \times \omega$.

 \begin{claim}
  There is $T' \subseteq T$ so that for every $(s,t) \in T'$, $p[T'_{(s,t)}]$ is cofinal in $X$.
 \end{claim}

 \begin{proof}
  Let $T' = \{ (s,t) : p[T_{(s,t)}] \text{ is cofinal in } X \}$. For every $(u,v) \in T \setminus T'$, we let $x_{u,v} \in X$ be such that $\forall y \in p[T_{(u,v)}] (x_{u,v} \subseteq^* y)$. The collection $\{ x_{u,v} : (u,v) \in T \setminus T' \}$ is countable and therefore there is $x \in X$ so that $x \subsetneq^* x_{u,v}$ for every $(u,v) \in T \setminus T'$. Now let $(s,t) \in T'$ be arbitrary and $x' \in X$ such that $x' \subseteq^* x$. As $p[T_{(s,t)}]$ is cofinal in $X$, there is $y \in p[T_{(s,t)}]$ so that $y \subseteq^* x'$. Say $(y,z) \in [T_{(s,t)}]$. For every $n \in \omega$, $(y \restriction n, z \restriction n) \in T'$ because else we get a contradiction to $y \subseteq^* x$. Thus $y \in p[T'_{(s,t)}]$.
 \end{proof}

 By the claim we can wlog assume that for every $(s,t) \in T$, $p[T_{(s,t)}]$ is cofinal in $X$. Now consider $T$ as a forcing notion (which is equivalent to Cohen forcing). The generic real will be a new element of $p[T]$ together with a witness. Let $\dot{c}$ be a name for the generic real. Notice that the statement that $p[T]$ is linearly ordered by $\subseteq^*$ is absolute. Thus for every $y \in X$ there is a condition $(s,t) \in T$ and $n \in \omega$ so that either $$ (s,t) \forces \dot{c} \subseteq y \setminus n $$
 or $$ (s,t) \forces y \subseteq \dot{c} \setminus n .$$

 The second option is impossible, because $p[T_{(s,t)}]$ is cofinal in $X$.
 We can thus find $(s,t)$, $n \in \omega$ and $Y \subseteq X$ cofinal in $X$, so that for every $y \in Y$, $(s,t) \forces \dot{c} \subseteq y \setminus n$. Let $(x,z) \in [T_{(s,t)}]$ be arbitrary. As $Y$ is cofinal in $X$, there is $y \in Y$ so that $y \subsetneq^* x$. But this clearly contradicts $(s,t) \forces \dot{c} \subseteq y \setminus n$.
\end{proof}

\begin{corollary}
 Every $\mathbf{\Sigma}^1_2$ inextendible linearly ordered tower has a cofinal subset of size $\omega_1$.
\end{corollary}

\begin{proof}
 Assume $X$ is $\mathbf{\Sigma}^1_2$. Then it is the union of $\omega_1$ many Borel sets (see e.g. \cite{YM1}). By Theorem~\ref{thm:noanalyticmlt} each of these Borel sets has a lower bound in $X$.
\end{proof}

Note that the above results can be applied similarly to inextendible linearly ordered subsets of $(\omega^\omega, \leq^*)$. 

\section{A ${\Pi}^1_1$ definable maximal tower in ${L}$}

In this section we will show how to construct in ${L}$ a maximal tower with a ${\Pi}^1_1$ definition. For this we apply the coding technique that has been developed by A. Miller in \cite{AM1} in order to show the existence of various nicely definable combinatorial objects in ${L}$.

Let $O$ be the set of odd and $E$ the set of even natural numbers.

\begin{lemma}
 Suppose $z \in 2^\omega$, $y \in [\omega]^\omega$ and $\seq{x_\alpha : \alpha < \gamma}$ is a tower, where $\gamma < \omega_1$, so that $\forall \alpha < \gamma ( \card{x_\alpha \cap O} = \omega \wedge \card{x_\alpha \cap E} = \omega)$. Then there is $x \in [\omega]^\omega$ so that $\forall \alpha < \gamma (x \subseteq^* x_\alpha)$, $\card{x \cap O} = \omega$, $\card{x \cap E} = \omega$, $z \leq_T x$ and $\card{y \cap \omega \setminus x} = \omega$.
\end{lemma}

\begin{proof}
 It is a standard diagonalization to find $x$ so that $\forall \alpha < \gamma (x \subseteq^* x_\alpha)$, $\card{x \cap O} = \omega$, $\card{x \cap E} = \omega$ and $\card{y \cap \omega \setminus x} = \omega$. We assume that $z$ is not eventually constant, else there is nothing to do. Now given $x$ find $\seq{n_i}_{i \in \omega}$ increasing in $x$ so that $n_i \in O$ iff $z(i) = 0$. Let $x' = \{ n_i : i < \omega \}$. Then $x'$ works.
\end{proof}

\begin{theorem}
\label{thm:coanalytictower}
 Assume ${V} = {L}$. Then there is a ${\Pi}^1_1$ definable maximal tower.
\end{theorem}

In the rest of the paper, $<_L$ will always stand for the canonical global ${L}$ well-order. Whenever $r \in 2^\omega$, we write $E_r \subseteq \omega^2$ for the relation defined by $$m E_r n \text{ iff } r(2^m 3^n) = 0.$$

If $E_r$ is a well-founded and extensional relation then we denote with $M_r$ the unique transitive $\in$-model isomorphic to $(\omega,E_r)$. Notice that $\{r \in 2^\omega : E_r \text{ is well-founded and extensional}\}$ is $\Pi^1_1$.

If $E_r$ is a well-order on $\omega$ then $\| r \|$ denotes the unique countable ordinal $\alpha$ so that $(\omega, E_r)$ is isomorphic to $(\alpha, \in)$. We also say that $r$ codes $\alpha$. The set of $r$ so that $E_r$ is a well-order is called $WO$. $WO$ is obviously $\Pi^1_1$. 

For any real $x \in 2^\omega$ we define $\omega_1^x$ to be the least countable ordinal which has no recursive code in $x$, i.e. the least ordinal $\alpha$ so that for any recursive function $r \colon 2^\omega \to 2^\omega$, $r(x)$ does not code $\alpha$. 

\begin{proof}[Proof of Theorem~\ref{thm:coanalytictower}]
Let $\seq{y_\xi : \xi < \omega_1}$ enumerate $[\omega]^\omega$ via the canonical well order of ${L}$. We will construct a sequence $\seq{\delta(\xi),z_\xi, x_\xi : \xi < \omega_1}$, where for every $\xi < \omega_1$:

 \begin{itemize}
  \item $\delta(\xi)$ is a countable ordinal
  \item $z_\xi \in 2^\omega \cap {L}_{\delta(\xi)+\omega}$
  \item $x_\xi \in [\omega]^{\omega} \cap {L}_{\delta(\xi)+\omega}$
 \end{itemize}

 The sequence is defined by the following requirements for each $\xi < \omega_1$:

 \begin{enumerate}
  \item $\delta(\xi)$ is the least ordinal $\delta$ greater than $\sup_{\nu < \xi} \delta(\xi)$ so that $y_\xi,\seq{\delta(\nu),z_\nu,x_\nu : \nu < \xi}  \in {L}_{\delta}$ and ${L}_{\delta}$ projects to $\omega$\footnote{This means that over ${L}_{\delta}$ there is a definable surjection to $\omega$. The set of such $\delta$ is unbounded in $\omega_1$.}.
  \item $z_\xi$ is the $<_L$ least code for the ordinal $\delta(\xi)$.
  \item $\seq{x_\nu : \nu < \xi}$ is a tower and $\forall \nu < \xi (\card{x_\nu\cap O}) = \omega \wedge \card{x_\nu \cap E} = \omega)$.
  \item $x_\xi$ is $<_L$ least so that $\forall \nu < \xi (x_\xi \subseteq^* x_\nu)$, $\card{x_\xi \cap O} = \omega$, $\card{x_\xi \cap E} = \omega$, $z_\xi \leq_T x$ and $\card{y_\xi \cap \omega \setminus x} = \omega$.
 \end{enumerate}

 Notice that $z_\xi$ and $x_\xi$ indeed can be found in ${L}_{\delta(\xi)+\omega}$ given that $y_\xi,\seq{x_\nu : \nu < \xi}  \in {L}_{\delta(\xi)}$, and that ${L}_{\delta(\xi)}$ projects to $\omega$. It is then straightforward to check that (1)-(4) uniquely determine a sequence $\seq{\delta(\xi),z_\xi, x_\xi : \xi < \omega_1}$ for which $\seq{x_\xi : \xi < \omega_1}$ is a maximal tower.

\begin{claim}
 $\{x_\xi : \xi < \omega_1 \}$ is a ${\Pi}^1_1$ subset of $2^\omega$.
\end{claim}
\begin{proof}
Let $\Psi(v)$ be the formula expressing that for some $\xi < \omega_1$, $v = \seq{\delta(\nu),z_\nu, x_\nu : \nu \leq \xi}$. More precisely, $\Psi(v)$ says that $v$ is a sequence $\seq{\rho_\nu, \zeta_\nu, \tau_\nu : \nu \leq \xi}$ of some length $\xi +1$, that satisfies the clauses (1)-(4) for every $\nu \leq \xi$.

The formula $\Psi(v)$ is absolute for transitive models of some finite fragment $\hbox{Th}$ of ZFC which holds at limit stages of the ${L}$ hierarchy. Namely we need absoluteness of the formula $\varphi_1(\xi,y)$ expressing that $y=y_\xi$, $\varphi_2(\delta,M)$ expressing that $M = {L}_{\delta}$ projects to $\omega$ and $\varphi_3(z,\delta)$ expressing that $z$ is the $<_L$ least code for $\delta$.

Moreover we have that $\seq{\delta(\nu),z_\nu, x_\nu : \nu \leq \xi} \in {L}_{\delta(\xi)+\omega}$ and that $${L}_{\delta(\xi)+\omega} \models \Psi(\seq{\delta(\nu),z_\nu, x_\nu : \nu \leq \xi})$$ for every $\xi < \omega_1$.

Now let $\Phi(r,x)$ be a formula expressing that $E_r$ is a well founded and extensional relation, $M_r \models \hbox{Th}$ and for some $v \in M_r$, $$M_r \models v \text{ is a sequence } \seq{\rho_\nu, \zeta_\nu, \tau_\nu : \nu \leq \xi} \wedge \Psi(v) \wedge \tau_\xi = x . $$

We thus have that $x = x_\xi$ for some $\xi < \omega_1$ iff $\exists r \in 2^\omega \Phi(r,x)$. $\Phi(r,x)$ can clearly be taken as a ${\Pi}^1_1$ formula.

For any $\xi < \omega_1$, the well order $\delta(\xi)$ is coded by $z_\xi$ and $z_\xi \leq_T x_\xi$. Thus $\delta(\xi) + \omega < \omega_1^{x_\xi}$ and there is $r \in {L}_{\omega_1^{x_\xi}}$ so that $M_r = {L}_{\delta(\xi)+\omega}$. In particular $$\exists r \in {L}_{\omega_1^{x_\xi}} \cap 2^\omega (\Phi(r,x_\xi)). $$

We get that $$\exists \xi < \omega_1 (x = x_\xi) \leftrightarrow \exists r \in {L}_{\omega_1^{x}} \cap 2^\omega (\Phi(r,x)). $$
The right hand side can be expressed by a ${\Pi}^1_1$ formula.

\end{proof}
\end{proof}

\begin{remark}
By Theorem~\ref{thm:towersubsetofL} the $\Pi^1_1$ tower constructed above is a subset of $L$. This implies that its definition will evaluate to the same set in any extension of $L$. 
\end{remark}

\begin{corollary}
 The existence of a coanalytic tower is consistent with the bounding number $\mathfrak{b}$ being arbitrarily large.
\end{corollary}

Recall that the bounding number is defined as the least size of an unbounded family in $(\omega^\omega, <^*)$. It is a natural lower bound for many other classical cardinal characteristics. 

\begin{proof}
 It is well known that finite support iterations of Hechler forcing for adding a dominating real preserve all ground model maximal towers to be maximal (see \cite{BD} for more details).
\end{proof}

\section{Indestructible Towers}

Remember that the pseudointersection number $\mathfrak{p}$ is the least cardinal $\kappa$ so that any set $\mathcal{F} \subseteq [\omega]^\omega$ with the finite intersection property and $\vert \mathcal{F} \vert < \kappa$ has a pseudointersection. $\mathcal{F}$ has the finite intersection property if for any $\mathcal{F}' \in [\mathcal{F}]^{<\omega}$, $\bigcap \mathcal{F}'$ is infinite. We obtain the following result. 

\begin{theorem}
Assume $\mathfrak{p} = \mathfrak{c}$. Let $\mathcal{P}$ be a collection of at most $\mathfrak{c}$ many proper posets of size $\mathfrak{c}$. Then there is a maximal tower indestructible by any $\mathbb{P} \in \mathcal{P}$.
\end{theorem}

Here $\mathfrak{c}$ is the size of the continuum. 

\begin{proof}
Let us call a $\mathbb{P}$ name $\dot x$ for a real a nice name whenever it has the form $\bigcup_{n \in \omega} \{ (p,\check n) : p \in A_{n} \}$ where the $A_n$'s are countable antichains in $\mathbb{P}$. Remember that when $\mathbb{P}$ is proper, then for any $\mathbb{P}$ name $\dot x$ for a real and any $p \in \mathbb{P}$, there is a nice name $\dot y$ and $q \leq p$ such that $q \forces \dot y = \dot x$. The number of nice $\mathbb{P}$ names is $\vert\mathbb{P}\vert^{\aleph_0}$.

Let us enumerate all pairs $\seq{ (\mathbb{P}_\alpha, p_\alpha, \dot y_\alpha): \alpha < \mathfrak{c}}$ where $p_\alpha \in \mathbb{P}_\alpha$, $\mathbb{P}_\alpha \in \mathcal{P}$ and $\dot y_\alpha$ is a nice $\mathbb{P}_\alpha$ name such that $p_\alpha \forces \dot y_\alpha \in [\omega]^\omega$.

We construct a tower $\seq{x_\alpha : \alpha < \mathfrak{c}}$ recursively. At step $\alpha$ we first choose a pseudointersection $x$ of $\seq{x_\xi : \xi < \alpha}$ (here we use $\alpha <\mathfrak{p}$). Next we partition $x$ into two disjoint infinite subsets $x^0, x^1$. Now note that $p_\alpha \forces_{\mathbb{P}_\alpha} (\dot y_\alpha \subseteq^* x^0 \wedge \dot y_\alpha \subseteq^* x^1)$ is impossible. Thus we find $i \in 2$ and $q_\alpha \leq p_\alpha$ such that $q_\alpha \forces_{\mathbb{P}_\alpha} \dot y_\alpha \not\subseteq^* x^i$. Let $x_\alpha = x^i$.

Now let $\dot x$ be an arbitrary $\mathbb{P}$ name for a real for some $\mathbb{P} \in \mathcal{P}$. We see easily that the set $D = \{ q \in \mathbb{P} : \exists \alpha < \mathfrak{c} ( q \forces \dot x \not\subseteq^* x_\alpha) \}$ is dense. Namely for any $p$ we find $(\mathbb{P}_\alpha, p_\alpha, \dot y_\alpha)$ where $p_\alpha \leq p$ and $p_\alpha \forces \dot x = \dot y_\alpha$. Then we have $q_\alpha \leq p$ with $q_\alpha \in D$. 
\end{proof}

\begin{definition}
 A forcing notion $(\mathbb{P},\leq)$ is Suslin if \begin{enumerate}
                                   \item $\mathbb{P} \subseteq 2^\omega$ is analytic,
                                   \item $\leq \subseteq 2^\omega \times 2^\omega$ is analytic,
                                   \item the incompatibility relation $\perp \subseteq 2^\omega \times 2^\omega$ is analytic (and in particular Borel). 
                                  \end{enumerate}

\end{definition}

The next thing we want to show is that (in $L$) for $\mathcal{P}$ the collection of all proper Suslin posets, we can get an indestructible maximal tower which is coanalytic.
\begin{theorem}
\label{thm:indestructibletower}
 (V=L) There is a $\Pi^1_1$ maximal tower indestructible by any proper Suslin poset. 
\end{theorem}

\begin{proof}
First let us note that there is a recursive map $f \colon \operatorname{Tree} \times [\omega]^\omega \to 2^\omega$, where $\operatorname{Tree}$ is the set of trees on $\omega \times \omega$, such that $f(T,y) \in \WO$ iff $\forall x \in p[T] ( \card{x \cap (\omega \setminus y)} = \omega)$. Fix this map $f$. 

 For the construction of our tower we now enumerate via the canonical well order of $L$ all trees $\seq{T_\alpha : \alpha < \omega_1}$ on $\omega \times \omega$. Now as in the proof of Theorem~\ref{thm:coanalytictower} we define a sequence $\seq{\delta(\xi),z_\xi, x_\xi : \xi < \omega_1}$ with
 
  \begin{itemize}
  \item $\delta(\xi)$ is a countable ordinal
  \item $z_\xi \in 2^\omega \cap {L}_{\delta(\xi)+\omega}$
  \item $x_\xi \in [\omega]^{\omega} \cap {L}_{\delta(\xi)+\omega}$
 \end{itemize}
 and the following properties: 
 
  \begin{enumerate}
  \item $\seq{x_\nu : \nu < \xi}$ is a tower and $\forall \nu < \xi (\card{x_\nu\cap O}) = \omega \wedge \card{x_\nu \cap E} = \omega)$.
  \item $\delta(\xi)$ is the least ordinal $\delta$ greater than $\sup_{\nu < \xi} \delta(\xi)$ so that 
  \begin{itemize}
   \item $\seq{\delta(\nu),z_\nu,x_\nu : \nu < \xi}, T_\xi  \in {L}_{\delta}$,
   \item there are disjoint pseudointersections $x^0,x^1 \in L_\delta$ of $\seq{x_\nu : \nu < \xi}$ both hitting $O$ and $E$ infinitely,
   \item either (a) there is $(x,w) \in [T_\xi] \cap L_\delta$ such that $x \subseteq^* x^0$ or (b) $f(T_\xi, x^0) \in \WO$, $\|f(T_\xi, x^0)\| < \delta$ and there is in $L_\delta$ a order preserving map $(\omega,E_{f(T_\xi, x^0)}) \to \|f(T_\xi, x^0)\|$, 
   \item and ${L}_{\delta}$ projects to $\omega$.
  \end{itemize}
  \item $z_\xi$ is the $<_L$ least code for the ordinal $\delta(\xi)$.
  \item $x_\xi$ is $<_L$ least so that $x_\xi \subseteq^* x^1$ or $x_\xi \subseteq^* x^0$ depending on whether (a) or (b) holds true, $\card{x_\xi \cap O} = \omega$, $\card{x_\xi \cap E} = \omega$ and $z_\xi \leq_T x_\xi$.
 \end{enumerate}

 As in the proof of Theorem~\ref{thm:coanalytictower} we see that this definition determines a tower $\seq{x_\xi : \xi < \omega_1}$ which is $\Pi^1_1$. 
 
 Now let us note the following for a proper Suslin poset $\mathbb{P}$. Whenever $\dot x$ is a nice $\mathbb{P}$ name for a real and $p \in \mathbb{P}$, then the set $$\{ z \in [\omega]^\omega : \exists q \leq p (n \in z \leftrightarrow q \not\forces n \in \omega \setminus \dot x) \}$$ is analytic ($q \not\forces n \in \omega \setminus \dot x$ iff $\exists r \in \dom \dot x [(r,n)\in \dot x \wedge r \not\perp q])$. 
 
 Thus for any $\mathbb{P},p\in \mathbb{P}$ and $\dot x$ a nice name there is $\alpha < \omega_1$ so that $$p[T_\alpha]= \{ z \in [\omega]^\omega : \exists q \leq p (n \in z \leftrightarrow q \not\forces n \in \omega \setminus \dot x) \}.$$
 
 Consider $x_\alpha$ and the respective disjoint sets $x^0$ and $x^1$ at stage $\alpha$ of the construction. There are two options: 
 
 \begin{enumerate}[(a)]
  \item There is $(x,w) \in [T_\alpha]$ such that $x \subseteq^* x^0$. In this case we have chosen $x_\alpha \subseteq^* x^1$ and there is $q \leq p$ so that $\vert \{n \in \omega : q\not\forces n \notin \dot x \} \cap x^1 \vert < \omega$. In particular $p \not\forces \dot x \subseteq^* x_\alpha$.
  \item Or $L_{\delta(\alpha)} \models$ ``$(\omega,E_{f(T_\xi,x^0)})$ \text{ is isomorphic to an ordinal}''. This means that $L \models$ ``$(\omega,E_{f(T_\xi,x^0)})$ \text{ is isomorphic to an ordinal}'' and this means that for any $x \in p[T_\alpha]$, $x$ has infinite intersection with $\omega \setminus x^0$. In this case we chose $x_\alpha \subseteq^* x^0$. Now if $q \leq p$ and $n \in \omega$ are arbitrary we can find $r \leq q$ and $m \geq n$ such that $r \forces m \in \dot x \setminus x_\alpha$. This means again that $p \not\forces \dot x \subseteq^* x_\alpha$. 
 \end{enumerate}
 
 Thus we have shown that for any proper Suslin poset $\mathbb{P}$, $\dot x$ an arbitrary $\mathbb{P}$ name for a real and $p \in \mathbb{P}$, $p \not\forces \dot x \text{ is a pseudointersection of } \langle x_\xi : \xi < \omega_1 \rangle$.

 \end{proof}

\section{$\omega_1$ and $\mathbf{\Sigma}^1_2$ definitions}

\begin{definition}
 Let $\mathcal{F}$ be a filter on $\omega$ containing all cofinite sets. Then Mathias forcing relative to $\mathcal{F}$ is the poset $\BbM(\mathcal{F})$ consisting of pairs $(s,F) \in [\omega]^{<\omega} \times \mathcal{F}$ such that $\max s < \min F$. The extension relation is defined by $(s,F) \leq (t,E)$ iff $t \subseteq s$, $F \subseteq E$ and $t \setminus s \subseteq E$.
\end{definition}

\begin{lemma}
\label{lem:destroysigma12tower}
 Assume that $X$ is a $\mathbf{\Sigma}^1_2$ definable subset of $[\omega]^\omega$, linearly ordered with respect to $\subseteq^*$. Then there is a ccc forcing notion $\Q$ consisting of reals so that for any transitive model ${V}' \supseteq {V}^\Q$ (with the same ordinals), the reinterpretation of $X$ in $V'$ is not an ilt in ${V}'$.
\end{lemma}

\begin{proof}
 As $X$ is $\mathbf{\Sigma}^1_2$, $X$ can be written as a union $\bigcup_{\xi < \omega_1} X_\xi$ of analytic sets. Namely whenever $X = p[Y]$ where $Y \subseteq [\omega]^\omega \times 2^\omega$ is coanalytic then $Y$ can be written as $\{ (x,w) : f(x,w) \in \WO \}$ for some fixed continuous function $f$ related to the definition of $Y$ (see \cite{YM1} for more details). Then $X_\xi$ is defined as $\{ x \in [\omega]^\omega : \exists w \in 2^\omega (\| f(x,w) \| = \xi) \}$.  
 
 Moreover we see that in any model $W \supseteq V$ where $\omega_1^W = \omega_1^V$, the reinterpretation of $X$ is the union of the reinterpretations of the $X_\xi$.  

 If $X$ has a pseudointersection $x$ in $V$, then $x$ will stay a pseudointersection of (the reinterpretation of) $X$ in any extension by absoluteness. The statement $ \forall y (y \notin X \vee x \subseteq^* y) $ is $\mathbf{\Pi}^1_2$. In this case let $\mathbb{Q}$ be the trivial poset.   

 If $X$ is inextendible in $V$, then for any $\xi < \omega_1$ there is $x_\xi \in X$ so that $\forall y \in X_\xi (x_\xi \subseteq^* y)$. As $X$ is linearly ordered with respect to $\subseteq^*$, $\{ x_\alpha : \alpha < \omega_1 \}$ generates a non-principal filter $\mathcal{F}$. Let $\Q = \BbM(\mathcal{F})$. Then in ${V}^\Q$ there is a real $x$ so that $x \subseteq^* x_\alpha$ for every $\alpha < \omega_1$. By absoluteness $\forall y \in X_\xi (x_\xi \subseteq^* y)$ will still hold true in $V^\mathbb{Q}$. In particular $\forall y \in X_\xi (x \subseteq^* y)$ will hold true for any $\xi  < \omega_1^V$.
 
As $\mathbb{Q}$ is ccc we have that $\omega_1^{V^\mathbb{Q}} = \omega_1^V$. This implies that $x$ is actually a pseudointersection of $X$ in $V^\mathbb{Q}$. Again, this will hold true in any extension.  
\end{proof}

\begin{theorem}
\label{thm:omega1}
If there is a $\Sigma^1_2$ ilt, then $\omega_1 = \omega_1^{L}$. More generally, the existence of a $\Sigma^1_2(x)$ ilt implies $\omega_1 = \omega_1^{L[x]}$.
\end{theorem}

\begin{proof}
 We only prove the first part as the rest follows similarly. 
 
 Suppose that $X$ is a $\Sigma^1_2$ ilt and $\omega_1^L < \omega_1$. Apply Lemma~\ref{lem:destroysigma12tower} to (the definition of) $X$ in $L$ to get the respective poset $\mathbb{Q}$ in $L$. As $\omega_1^L < \omega_1$, $V \models \vert \mathcal{P}(\omega) \cap L\vert = \omega$. But this means there is a $\mathbb{Q}$ generic $x \in V$ over $L$. $L[x] \subseteq V$, thus by Lemma~\ref{lem:destroysigma12tower} $X$ has a pseudointersection in $V$, contradicting our assumption. 
\end{proof}

\begin{remark}
We think that the proofs of Lemma~\ref{lem:destroysigma12tower} and Theorem~\ref{thm:omega1} showcase something interesting about Schoenfield absoluteness. Recall that Schoenfield's absoluteness theorem says that $\Sigma^1_2$ formulas are absolute between any inner models $W \subseteq W'$, but it does not say anything about the relationship between $\omega_1^W$ and $\omega_1^{W'}$. In fact in many applications of $\Sigma^1_2$ absoluteness $W$ and $W'$ have the same $\omega_1$ (e.g. when $W'$ is a ccc or proper forcing extension of $W$). But in this case it can be deduced directly from analytic absoluteness and the representation of $\Sigma^1_2$ sets as the same $\omega_1$ union of analytic set in any extension with the same $\omega_1$. The reason is that the existential quantifier $\exists \alpha < \omega_1$ stays the same. So the full strength of Schoenfield absoluteness is only needed in the case where $\omega_1^W$ is countable in $W'$ and this is the case that we crucially used in the proof of Theorem~\ref{thm:omega1}. 
\end{remark}

We also want to remark that the proofs of Lemma~\ref{lem:destroysigma12tower} and Theorem~\ref{thm:omega1} are very general and can be applied to many other maximal combinatorial families. For example A. T\"ornquist has shown the following theorem in \cite{AT1}, using a similar argument. 

\begin{theorem}
If there is a $\Sigma^1_2$ mad family, then $\omega_1 = \omega_1^{L}$. More generally, the existence of a $\Sigma^1_2(x)$ mad family implies $\omega_1 = \omega_1^{L[x]}$.
\end{theorem}

The argument for maximal independent families is a bit different. Let us recall the definition of a maximal independent family. 

\begin{definition}
 A set $X \subseteq [\omega]^\omega$ is called independent if for any $F \in [X]^{<\omega}$ and $G \in [X]^{<\omega}$ where $F \cap G = \emptyset$, $\bigcap_{x \in F} x \cap \bigcap_{y \in G} (\omega \setminus y)$ is infinite. An independent family is called maximal if it is maximal under inclusion.
\end{definition}

The set $\bigcap_{x \in F} x \cap \bigcap_{y \in G} (\omega \setminus y)$ is often denoted $\sigma(F,G)$. We will also use this notation below. Note that an independent family $X$ is not maximal iff there is a real $x$ so that $x \cap \sigma(F,G)$ and $(\omega \setminus x) \cap \sigma(F,G)$ are infinite for all $F,G \in [X]^{<\omega}$ where $F \cap G = \emptyset$. Such a real will be called independent over $X$. 

We obtain the following result.
\begin{theorem}
\label{thm:independent}
If there is a $\Sigma^1_2$ maximal independent family, then $\omega_1 = \omega_1^{L}$. More generally, the existence of a $\Sigma^1_2(x)$ maximal independent family implies $\omega_1 = \omega_1^{L[x]}$.
\end{theorem}

In \cite{AM1} Miller basically proved that a Cohen real is independent over any ground model coded analytic independent family. He did not put his theorem in these words, so before we go on let us repeat his argument. 

\begin{lemma}[{\cite[Proof of Theorem 10.28]{AM1}}]
\label{lem:independent}
 Let $\varphi(x)$ be a $\mathbf{\Sigma}^1_1$ formula defining an independent family and let $c$ be a Cohen real over $V$. Then in $V[c]$, $c$ is independent over the family defined by $\varphi(x)$.  
\end{lemma}

\begin{proof}
Let $X$ denote the set $\{x \in [\omega]^\omega : \varphi(x) \}$ in any model extending $V$. Note that in any model $X$ is an independent family by Schoenfield absolutness. 
 Let $$K = \{x \in [\omega]^\omega : \exists F \in [X]^{<\omega} \exists G \in [X]^{<\omega} (F \cap G = \emptyset \wedge \vert \sigma(F,G) \cap x \vert< \omega)  \} $$
 
 and 
 
 $$H = \{x \in [\omega]^\omega : \exists F \in [X]^{<\omega} \exists G \in [X]^{<\omega} (F \cap G = \emptyset \wedge \vert \sigma(F,G) \cap (\omega \setminus x) \vert< \omega)  \} .$$
 
 These sets are both analytic. Note that $x$ is independent over $X$ iff $x \notin H \cup K$. To show that any Cohen real $c$ is independent over $X$, i.e. $c \notin H \cup K$ it suffices to prove that $H$ and $K$ are meager. Why? When $H \cup K$ is meager then there is a meager $F_\sigma$ set $C$ so that $H \cup K \subseteq C$ and this statement is absolute ($\forall x (x \in H \cup K \rightarrow x \in C)$). As $c$ is Cohen, $V[c] \models c \notin C$ and thus $V[c] \models c \notin H \cup K$ which implies that in $V[c]$, $c$ is independent over $X$. 
 
 So let us prove: 
 \begin{clm}
  $K$ and $H$ are meager. 
 \end{clm}

 \begin{proof}
 Suppose e.g. that $H$ is nonmeager. The argument for $K$ will follow similarly. Because $H$ is analytic it has the Baire property and is thus comeager somewhere. It is well known and easy to see that any comeager set contains a perfect set of almost disjoint reals. So let $P \subseteq H$ be a perfect almost disjoint family. For each $x \in P$ we have $F_x$ and $G_x$ so that $\sigma(F_x,G_x) \subseteq^* x$. By the Delta system lemma, there is a set $S \in [P]^{\omega_1}$ and $R \in [P]^{<\omega}$ so that $$\forall x \neq y \in S ((F_x \cup G_x) \cap (F_y \cup G_y) = R).$$ For any $x\in S$ we define $R_x^0 = R \cap F_x$ and $R_x^1 = R \cap G_x$. As $S$ is uncountable there is an uncountable $S' \subseteq S$ so that $$\forall x,y \in S' (R_x^0 = R_y^0 \wedge R_x^1 = R_y^1).$$
 
 But now note that for any $x \neq y \in S'$, $F_x \cap G_y = (R \cap F_x) \cap (R \cap G_y) = R_x^0 \cap R_y^1 = R_x^0 \cap R_x^1 = \emptyset$. By symmetry we also have that $F_y \cap G_x = \emptyset$ and this implies that $$(F_x \cup F_y) \cap (G_x \cup G_y) = \emptyset.$$ In particular we can form $\sigma(F_x \cup F_y, G_x \cup G_y)$. By choice of $F_x,G_x,F_y,G_y$ we have that $$\sigma(F_x \cup F_y, G_x \cup G_y) \subseteq^* x \cap y =^* \emptyset$$ as $P$ was an almost disjoint family. But this contradicts the independence of $X$.  
 \end{proof}
 
\end{proof}

\begin{proof}[Proof of Theorem~\ref{thm:independent}]
Assume $X$ is a $\Sigma^1_2$ maximal independent family. Then in $L$, $X$ is also independent and it can be written as a union $\bigcup_{\xi < \omega_1^L} X_\xi$ of analytic sets $X_\xi$. As $\omega_1^L < \omega_1$, there is a Cohen real $c$ over $L$. We have that $\omega_1^{L[c]} = \omega_1^L$ and in $L[c]$, $X$ still corresponds to the union $\bigcup_{\xi < \omega_1^L} X_\xi$. By the above lemma $c$ is independent over all the $X_\xi$ so in particular $c$ is independent over $X$. This statement is $\Pi^1_2$ and thus absolute between any inner models containing $c$. In particular in $V$, $X$ is not maximal. 
\end{proof}

\begin{theorem}
\label{thm:hamel}
If there is a $\Sigma^1_2$ Hamel basis of $\mathbb{R}$, then $\omega_1 = \omega_1^{L}$. More generally, the existence of a $\Sigma^1_2(x)$ Hamel basis of $\mathbb{R}$ implies $\omega_1 = \omega_1^{L[x]}$.
\end{theorem}

A Hamel basis of $\mathbb{R}$ is a maximal set of linearly independent reals over the rationals $\mathbb{Q}$. Again it was Miller who first showed that a Cohen real in $\mathbb{R}$ is independent over any ground model coded analytic linearly independent family.   

\begin{lemma}[{\cite[Proof of Theorem 9.25]{AM1}}]
\label{lem:hamel}
 Assume $A\subseteq \mathbb{R}$ is an analytic set of reals that are linearly independent over the field of rationals. Assume $c \in \mathbb{R}$ is a Cohen real over $V$. Then in $V[c]$, $c$ is linearly independent over (the reinterpretation of) $A$.     
\end{lemma}

\begin{proof}
 We assume that $A \neq \emptyset$, else the argument is trivial. Let $x \in A \cap V$ be arbitrary. Suppose that $U \forces$ ``$\dot c$ is not independent over $A$'' where $U \subseteq \mathbb{R}$ is some basic open set. Say $$U \forces \exists x_0,\dots,x_n \in A \exists q_0,\dots q_n \in \mathbb{Q} (\dot c = q_0x_0 + \dots + q_nx_n)$$ for some $n \in \omega$. 
 Now let $c \in U$ be Cohen over $V$ and $x_0,\dots,x_n, \in A, q_0,\dots q_n \in \mathbb{Q}$ so that $$c = q_0x_0 + \dots + q_nx_n.$$ Let $s$ be a small enough rational number so that $c + s x \in U$. Remember that, as $x \in V$, $c + s x$ is also a Cohen real over $V$. Thus let $y_0, \dots, y_n \in A, r_0, \dots, r_n \in \mathbb{Q}$ so that  
 
 $$c + sx = r_0y_0 + \dots + r_ny_n.$$
 
 But now we have that $$r_0y_0 + \dots + r_ny_n - (q_0x_0 + \dots + q_nx_n) = sx$$ and so $A$ is not linearly independent in $V[c]$. But this is impossible by absoluteness. 
\end{proof}

\begin{proof}[Proof of Theorem~\ref{thm:hamel}]
Same as the proof of Theorem~\ref{thm:independent}.
\end{proof}

For ultrafilters the proof is not much different. It will appear in \cite{JS1}. 

\begin{theorem}
\label{thm:ultrafilter}
If there is a $\Sigma^1_2$ ultrafilter, then $\omega_1 = \omega_1^{L}$. More generally, the existence of a $\Sigma^1_2(x)$ ultrafilter implies $\omega_1 = \omega_1^{L[x]}$.
\end{theorem}

We want to remark the ideas above can also be used to show that under Martin's Axiom none of the families above have $\mathbf{\Sigma}^1_2$ witnesses.  

\begin{theorem}
 MA($\omega_1$) implies that there is no $\mathbf{\Sigma}^1_2$ ilt, mad family, maximal independent family, Hamel basis or ultrafilter. 
\end{theorem}

\begin{proof}
 For mad families this was proven in \cite{AT1}. For ilt's Theorem~\ref{thm:towersubsetofL} is enough. For ultrafilters it suffices to note that under MA($\omega_1$) every $\mathbf{\Sigma}^1_2$ set is Lebesgue measurable (see \cite{AKA}) and an ultrafilter cannot be Lebesgue measurable. The argument for independent families and Hamel bases is the same. Write $X = \bigcup_{\xi < \omega_1} B_\xi$ where the $B_\xi$'s are analytic. Let $M$ be an elementary submodel of size $\omega_1$ containing all the parameters defining the $B_\xi$'s. Then let $c \in V$ be Cohen over $M$ and use Lemma~\ref{lem:independent} or Lemma~\ref{lem:hamel} to conclude that $c$ is independent or linearly independent over $X$.
\end{proof}

\section{Solovay's model}

In this section we prove the following result. 

\begin{theorem}
\label{thm:solovay}
 There is no ilt in Solovay's model. 
\end{theorem}

Let us review some basics about Solovay's model. A good presentation of Solovay's model can be found in \cite[Chapter 26]{TJ1}. Assuming $\kappa$ is an inaccessible cardinal in the constructible universe $L$ we first form an extension $V$ of $L$ in which $\omega_1 = \kappa$ using the L\'evy collapse (see again \cite[Chapter 26]{TJ1}). Then we let $W \subseteq V$ consist of all sets which are hereditarily definable from ordinals and reals as the only parameters. $W$ is then called Solovay's model. The only facts that we use about $W$ are listed below and are well-known.  

Suppose $a \in 2^\omega \cap W$ is arbitrary, then

 \begin{enumerate}
  \item for every poset $\mathbb{P} \in H(\kappa)^{L[a]}$, there is a $\mathbb{P}$ generic filter over $L[a]$ in $W$, 
  \item whenever $x \in 2^\omega \cap W$, there is a poset $\mathbb{P} \in H(\kappa)^{L[a]}$, $\sigma \in H(\kappa)^{L[a]}$ a $\mathbb{P}$ name and $G \in W$ a $\mathbb{P}$ generic over $L[a]$ so that $x = \sigma[G]$.
 \end{enumerate}

 Suppose $X \in \mathcal{P}(2^\omega) \cap W$. Then there is $a \in 2^\omega \cap W$ and a formula $\varphi(x)$ in the language of set theory using only $a$ and ordinals as parameters so that 
 
 \begin{enumerate}
 \setcounter{enumi}{2}
  \item for any poset $\mathbb{P} \in H(\kappa)^{L[a]}$, $\sigma \in H(\kappa)^{L[a]}$ a $\mathbb{P}$ name and $G \in W$, $\mathbb{P}$ generic over $L[a]$, $$ \sigma[G] \in X \leftrightarrow \exists p \in G (p \forces \varphi(\sigma)) .$$
 \end{enumerate}

Until the end of the section we are occupied with proving Theorem~\ref{thm:solovay}.
To prove Theorem~\ref{thm:solovay}, assume that $X \in \mathcal{P}(2^\omega) \cap W$ is linearly ordered with respect to $\subseteq^*$. We will show that $X$ cannot be an ilt. Let $a \in 2^{\omega} \cap W$ and $\varphi(x)$ be as in (3). To simplify notation we will assume that $a \in L$ and thus $L[a] = L$. From now on let us work in $L$. 

\begin{lemma}
 Let $\mathbb{P} \in H(\kappa)$, $p \in \mathbb{P}$ and $\sigma$ a $\mathbb{P}$ name so that $p \forces \varphi(\sigma)$. Then there is $p_0, p_1 \leq p$ and $n \in \omega$ so that for any $m \geq n$, $$\exists r \leq p_0 (r \forces m \in \sigma) \rightarrow p_1 \forces m \in \sigma.$$ 
\end{lemma}

\begin{proof}
 Consider $\mathbb{P} \times \mathbb{P} \in H(\kappa)$ and $\sigma_0$ and $\sigma_1$ the $\mathbb{P} \times \mathbb{P}$ names so that whenever $G_0 \times G_1$ is $\mathbb{P} \times \mathbb{P}$ generic over $V$ then $\sigma_0[G_0 \times G_1] = \sigma[G_0]$, $\sigma_1[G_0 \times G_1] = \sigma[G_1]$. 
 
 Note that $(p,p) \forces \varphi(\sigma_0) \wedge \varphi(\sigma_1)$, because whenever $G_0 \times G_1$ is $\mathbb{P} \times \mathbb{P}$ generic over $V$ with $(p,p) \in G_0 \times G_1$ then $G_0$ and $G_1$ are $\mathbb{P}$ generic over $V$ with $p \in G_0, G_1$. But then there must be $(p_0,p_1) \leq (p,p)$ and $n \in \omega$ so that either, $$(p_0,p_1) \forces \sigma_0 \setminus n \subseteq \sigma_1$$ or $$(p_0,p_1) \forces \sigma_1 \setminus n \subseteq \sigma_0.$$ 
 
Say wlog that $(p_0,p_1) \forces \sigma_0 \setminus n \subseteq \sigma_1$. Note that whenever $\exists r_0 \leq p_0 (p_0 \forces m \in \sigma)$ for some $m \geq n$ then $p_1 \forces m \in \sigma$. Suppose this was not the case. Then there is $r_1 \leq p_1$ so that $r_1 \forces m \notin \sigma$. But then $(r_0,r_1) \forces \exists m \geq n (m \in \sigma_0 \wedge m \notin \sigma_1)$ which is a contradiction to $(r_0,r_1) \leq (p_0,p_1)$.  
\end{proof}

Still in $L$, let $\langle\mathbb{P}_\xi, p_\xi, \sigma_\xi : \xi < \kappa \rangle$ enumerate all triples $\langle \mathbb{P}, p, \sigma\rangle $, where $\mathbb{P} \in H(\kappa)$, $p \in \mathbb{P}$ and $\sigma \in H(\kappa)$ is a $\mathbb{P}$ name so that $p \forces \varphi(\sigma)$. This is possible as $\vert H(\kappa) \vert = \kappa$. 
 
For every $\xi < \kappa$ we find $p_\xi^0, p_\xi^1 \leq p_\xi$ in $\mathbb{P}_\xi$ and $n \in \omega$ so that for every $m \geq n$ $$\exists r \leq p_\xi^0 (r \forces m \in \sigma_\xi) \rightarrow p_\xi^1 \forces m \in \sigma_\xi.$$ 

Let $x_\xi = \{ m \in \omega : p_\xi^1 \forces m \in \sigma_\xi \}$ for every $\xi < \kappa$. 

\begin{clm}
 $\{x_\xi : \xi < \kappa \}$ has the finite intersection property. 
\end{clm}

\begin{proof}[Proof of Claim.]
Suppose $x_{\xi_0}, \dots x_{\xi_{k-1}}$ are such that $\bigcap_{i<k} x_{\xi_i}$ is finite, say $\bigcap_{i<k} x_{\xi_i} \subseteq n$. Consider the poset $\mathbb{Q} = \prod_{i<k} \mathbb{P}_{\xi_i} \in H(\kappa)$, $(p_{\xi_0}^0, \dots,p_{\xi_{k-1}}^0 ) \in \mathbb{Q}$ and for every $i<k$, $\sigma_i$ the $\mathbb{Q}$ name so that whenever $(G_0,\dots, G_{k-1})$ is $\mathbb{Q}$ generic then $\sigma_i[G_0 \times \dots \times G_{k-1}] = \sigma_{\xi_i}[G_i]$. 

We have that $(p_{\xi_0}^0, \dots,p_{\xi_{k-1}}^0 ) \forces \varphi(\sigma_0) \wedge \dots \wedge \varphi(\sigma_{k-1})$ and thus, as $X$ has the finite intersection property, there is $m \geq n$ and $(r_0, \dots, r_{k-1}) \leq (p_{\xi_0}^0, \dots,p_{\xi_{k-1}}^0)$ so that $$(r_0, \dots, r_{k-1}) \forces m \in \bigcap_{i<k} \sigma_i.$$ 

But this means that $r_i \forces m \in \sigma_i$ and thus $m \in x_{\xi_i}$ for each individual $i$. This contradicts $\bigcap_{i<k} x_{\xi_i} \subseteq n$ as $m \geq n$.
\end{proof}

Let $\mathcal{F}$ be the filter generated by $\{x_\xi : \xi < \kappa \}$. We have that $\mathcal{F}\in \mathcal{P}([\omega]^\omega)$ and thus $\mathcal{F} \in H(\kappa)$. Moreover we have that $\mathbb{M}( \mathcal{F}) \in H(\kappa)$. Thus in $W$ there is $y \in [\omega]^\omega$ a $\mathbb{M}( \mathcal{F})$ generic real over $L$. 

\begin{clm}
 For every $x \in X$, $y \subseteq^* x$. In particular $X$ is not an ilt.  
\end{clm}

\begin{proof}[Proof of Claim]
Let $x \in X$ be arbitrary. Then we have in $L$ a poset $\mathbb{P} \in H(\kappa)$ and a $\mathbb{P}$ name $\sigma$ so that there is in $W$ a $\mathbb{P}$ generic $G$ over $V$ so that $x = \sigma[G]$. Moreover there is $p \in G$ so that $p \forces \varphi(\sigma)$. 

It suffices to show that there is some $\xi < \kappa$ and $q \in G$ so that $q \forces x_\xi \subseteq^* \sigma$. To see this we simply show that the set of conditions $q \in \mathbb{P}$ so that $\exists \xi < \kappa (q \forces x_\xi \subseteq^* \sigma)$ is dense below $p$. To show this fix $p' \leq p$ arbitrary. Let $\xi$ be such that $\langle \mathbb{P}, p', \sigma \rangle = \langle \mathbb{P}_\xi, p_\xi, \sigma_\xi \rangle$. But then $p_\xi^1 \leq p_\xi$ and $p_\xi^1 \forces x_\xi \subseteq^* \sigma_\xi$.  
\end{proof}

This finishes the proof of Theorem~\ref{thm:solovay}.

\section{$\mathbf{\Sigma}^1_2$ vs $\mathbf{\Pi}^1_1$}

\begin{theorem}
 Any $\Pi^1_1(x)$ ilt contains a $\Pi^1_1(x)$ maximal tower. 
\end{theorem}

\begin{proof}
 We are going to prove the statement only for lightface $\Pi^1_1$ as everything will relativize. So let $X$ be a $\Pi^1_1$ ilt. 
 
 \begin{clm}
  $X \cap L$ is cofinal in $X$ (where $L$ is the constructible universe). 
 \end{clm}

 \begin{proof}
  By Theorem~\ref{thm:omega1} we have that $\omega_1 = \omega_1^L$ must be the case. Thus $X$ can be written as a union $\bigcup_{ \xi < \omega_1} X_\xi$ of analytic sets $X_\xi$ which are coded in $L$ (see the proof of Lemma~\ref{lem:destroysigma12tower}). Note that $X \cap L$ is an ilt in $L$ by a downwards absoluteness argument. This implies that for every $\xi < \omega_1$ there is $x \in L \cap X$ which is a pseudointersection of $X_\xi$. The statement ``$x$ is a pseudointersection of $X_\xi$'' is absolute. Thus $X \cap L$ is indeed cofinal in $X$. 
 \end{proof}

 We may now work entirely in $L$, assume $X \in L$ and construct a $\Pi^1_1$ tower that is cofinal in $X$ (which implies that it is cofinal in $X$ as interpreted in $V$).

 Recall that $C_1 = \{ x \in 2^\omega : x \in L_{\omega_1^x} \}$ is the largest thin $\Pi^1_1$ set and for any $y$, $C_1(y) = \{ x \in 2^\omega : x \in L_{\omega_1^x}[y] \}$ is the largest thin $\Pi^1_1(y)$ set (see e.g \cite{YM1}).  
 
 \begin{clm}
  $C_1 \cap X$ is cofinal in $X$.
 \end{clm}
 
 \begin{proof}
  Suppose not, i.e there is $y \in X$ so that $C_1 \cap X \subseteq \{x \in [\omega]^\omega : y \subseteq^* x \}$. The set $ Y := \{ x \in X : x \subseteq^* y \}$ is $\Pi^1_1(y)$. We distinguish between two cases. 
  
  \begin{itemize}
   \item Case 1: $Y$ is thin. Then $Y \subseteq C_1(y)$. Moreover $\{\omega_1^z : z \in Y \}$ is unbounded in $\omega_1 = \omega_1^L$ (if $r$ is recursive such that $X = \{ x : r(x) \in \WO\}$ and $\delta$ bounds $\{\omega_1^z : z \in Y \}$, then $\{x : \|r(x)\| < \delta \}$ is a Borel subset of $X$ that is cofinal in $X$ which is impossible). But note that there is $\alpha < \omega_1$ large enough so that $L_\alpha [y] = L_\alpha$ and further $L_\beta [y] = L_\beta$ for any $\beta \geq \alpha$. So if $\omega_1^z > \alpha$ and $z \in Y$ then $z \in L_{\omega_1^z}[y] = L_{\omega_1^z}$. This is a contradiction to our assumption.  
   \item Case 2: $Y$ contains a perfect set $P$. By a theorem of Martin and Friedman (see \cite{M}), $P$ contains reals in any $\Delta^1_1$ degree above the degree of some $d \in 2^\omega$. The set $C_1$ is unbounded in the $\Delta^1_1$ degrees of $L$ (see \cite{YM1}) and is closed under $\Delta^1_1$ bi-reducibility. Thus there is $z \in P$ so that $z \in C_1$. Again we get a contradiction to our assumption.   
  \end{itemize}
 \end{proof}

 From now on we may assume that $X \subseteq C_1$. In the next step we will thin out $X$ even further. For each $x \in X$ let $\alpha_x < \omega_1^x$ be such that $x \in L_{\alpha_x}$ (this is possible as $x \in L_{\omega_1^x}$). Further let $r_x \colon [\omega]^\omega \to 2^\omega$ be recursive so that $\alpha_x = \| r_x(x) \|$. As there are only countably many recursive functions, there is one $r$ so that the set $\{ x \in X : r_x = r \}$ is cofinal in $X$. Fix such an $r$. Let $$Y := \{ x \in X : x \in L_{\| r(x) \|} \}. $$  
 
 $Y$ is a $\Pi^1_1$ cofinal subset of $X$. Thus let $s \colon [\omega]^\omega \to 2^\omega$ be a recursive function such that $Y = \{ x \in [\omega]^\omega : s(x) \in \WO \}$. We define the following well-order $\lhd$ on $Y$: $$x \lhd y \leftrightarrow \| s(x) \| < \| s(y) \| \vee (\| s(x) \| = \| s(y) \| \wedge x <_L y) $$
 
Let $\varphi_0(w,v)$ be a $\Sigma^1_1$ formula expressing that $(\omega,E_w)$ is properly embedable into $(\omega,E_v)$ and let $\varphi_1(w,v)$ be a $\Sigma^1_1$ formula expressing that $(\omega,E_w)$ is isomorphic to $(\omega,E_v)$. Moreover let $\psi(x,y)$ be a $\Sigma^1_1$ formula so that whenever $y \in Y$ and $x$ is arbitrary then $\psi(x,y)$ is equivalent to $x \in L_{\|r(y)\|} \wedge L_{ \|r(y)\|} \models x <_L y$. Let $\chi (x,y)$ be the $\Sigma^1_1$ formula $$\varphi_0(s(x),s(y)) \vee (\varphi_1(s(x),s(y)) \wedge \psi(x,y)).$$

We see that when $y \in Y$ then $$\chi(x,y) \leftrightarrow x \lhd y .$$ In particular, when $y \in Y$ then $\chi(x,y)$ implies that $x \in Y$. 

Finally we define $T := \{y \in Y : \forall x \lhd y (y \subseteq^* x) \}$. We have that $T$ is $\Pi^1_1$ as $y \in T$ iff $$y \in Y \wedge \forall x (\neg \chi(x,y) \vee y \subseteq^* x) .$$

$T$ is obviously a tower as the order $^*\supseteq$ on $T$ coincides with $\lhd$. $T$ is cofinal in $Y$ as for any $x \in Y$ if we let $y$ be $\lhd$ least in $Y$ so that $y \subseteq^* x$ then $y \in T$.

\end{proof}

\begin{theorem}
 The existence of a $\Sigma^1_2(x)$ ilt implies the existence of a $\Pi^1_1(x)$ ilt. 
\end{theorem}

\begin{proof}
Let $X$ a $\Sigma^1_2$ ilt. As in the proof above we can show that $X \cap L$ is cofinal in $X$ (and this uses $\omega_1 = \omega_1^L$). So as $[\omega]^\omega \cap L$ is $\Sigma^1_2$ we may just assume that $X \subseteq L$. Let $\varphi(x,w)$ be $\Pi^1_1$ such that $x \in X$ iff $\exists w \varphi(x,w)$. Using $\Pi^1_1$ uniformization we can further assume that $x \in X$ iff $\exists! w \varphi(x,w)$. 




The idea will now be to get a linearly ordered tower that basically consists of $x \in X$ together with their unique witness $w$. To do this we have to introduce some notation. 

\begin{itemize}
 \item For $y \subseteq [\omega \times \omega]^\omega$, we write $y_n$ for $y$'s $n$'th vertical section. 
 \item For $x \in [\omega]^\omega$, we write $x(n)$ for the $n$'th element of $x$.
\end{itemize}

We now define the new ilt $Y$ which lives on $\omega \times \omega$. A set $y \in [\omega \times \omega]^\omega$ is in $Y$ iff the following hold true: 

\begin{enumerate}
 \item For every $n \geq 1$, $y_n = y_0 \setminus y_0(n)$ or $y_n = y_0 \setminus y_0(n+1)$. 
 \item If $w \in 2^\omega$ is such that $w(n) = \begin{cases}
                                                 0 \text{ if } y_{n+1} = y_0 \setminus y_0(n+1)\\
                                                 1 \text{ if } y_{n+1} = y_0 \setminus y_0(n+2)
                                                \end{cases}
$ then $\varphi(y_0,w)$ and in particular $y_0 \in X$. 
\end{enumerate}

\begin{enumerate}[(i)]
 \item $Y$ is $\Pi^1_1$: Checking whether $y \in [\omega \times \omega]^\omega$ is as described in (1) is $\Delta^1_1$. Checking whether for the function $w \in 2^\omega$ as in (2), $\varphi(y_0,w)$ holds true is $\Pi^1_1$.  
 \item $Y$ is linearly ordered by $\subseteq^*$: Let us note first that whenever $x \subsetneq^* y$ then eventually $x(n) > y(n)$. Why is this the case? As $x \subsetneq^* y$ (so $x \neq^* y$), there is a big enough $n \in \omega$ so that $\forall m \geq n (\vert y \cap x(m) \vert > m$. But this means that $x(m) > y(m)$ for all $m \geq n$. 
 
 Now let's assume that $x \neq y \in Y$ and without loss of generality that $x_0 \subsetneq^* y_0$. By the observation above there is an $n$ so that for every $m \geq n$, $x_0(m) > y_0(m)$ and $x_0(m) \in y_0$. But this also means that $\forall m \geq n$, $$x_m \subseteq x_0 \setminus x_0(m) \subseteq y_0 \setminus y_0(m+1) \subseteq y_m.$$ In particular $x_m \subseteq y_m$ for $m \geq n$. For $k < n$ we have that $x_k \subsetneq^* y_k$. Thus all together we have that $x \subsetneq^* y$. 
 
 \item $Y$ has no pseudointersection: Suppose $z$ is a pseudointersection of $Y$. If there is $n \in \omega$ so that $\vert z_n \vert = \omega$, then $z_n$ is a pseudointersection of $X$. Else let $x = \{ \min z_n : n \in \omega \wedge z_n \neq \emptyset \}$. It is easy to see that $x$ must be infinite (else $z$ would not be $\subseteq^*$ below any member of $Y$). We claim that $x$ is a pseudointersection of $X$. Namely let $y_0 \in X$ be arbitrary where $y \in Y$. As $z \subseteq^* y$, there is an $n$ so that $\forall m \geq n ( z_m \neq \emptyset \rightarrow (m,\min z_m) \in y )$. This means in particular that $\forall m \geq n (z_m \neq \emptyset \rightarrow \min z_m \in y_0)$.

\end{enumerate}

\end{proof}

\begin{theorem}
 \label{thm:bigequivalence}
 The following are equivalent:
 \begin{enumerate}
  \item There is a $\Sigma^1_2(x)$ ilt.
  \item There is a $\Pi^1_1(x)$ ilt.
  \item There is a $\Pi^1_1(x)$ maximal tower.
  \item There is a $\Sigma^1_2(x)$ maximal tower.
 \end{enumerate}
 
 \end{theorem}
 
 \begin{proof}
  We have shown above that $(1) \rightarrow (2) \rightarrow (3)$. $(3) \rightarrow (4) \rightarrow (1)$ are trivial from the definitions. 
 \end{proof}


\begin{thebibliography}{00}

 \bibitem{BD}
 J. Baumgartner, P. Dordal
{\emph{Adjoining dominating functions.}}
J. Symbolic Logic 50 (1985), no. 1, 94-101.


\bibitem{BFK1}
J. Brendle, V. Fischer, Y. Khomskii
{\emph{Definable maximal independent families}}
accepted at the Transactions of the American Mathematical Society. 

\bibitem{BK1}
J. Brendle, Y. Khomskii
{\emph{Mad families constructed from perfect almost disjoint families.}}
J. Symbolic Logic 78 (2013), no. 4, 1164-1180. 

\bibitem{FFK1}
V. Fischer, S. D. Friedman, Y. Khomskii
{\emph{Co-analytic mad families and definable wellorders.}}
Arch. Math. Logic 52 (2013), no. 7-8, 809-822. 

\bibitem{FS1}
V. Fischer, D. Schrittesser
{\emph{Sacks indestructible eventually different families.}}
Submitted.


\bibitem{FST1}
V. Fischer, D. Schrittesser, A. T\"ornquist
{\emph{A co-analytic Cohen-indestructible maximal cofinitary group.}}
J. Symb. Log. 82 (2017), no. 2, 629-647. 

\bibitem{FT1}
V. Fischer, A. T\"ornquist
{\emph{A co-analytic maximal set of orthogonal measures.}}
Journal of Symbolic Logic, 75, 4, pp. 1403-1414.


\bibitem{HS1}
H. Horowitz, S. Shelah
{\emph{A Borel maximal cofinitary group.}}
arXiv:1610.01344.

\bibitem{TJ1}
T. Jech
{\emph{Set theory.}}
The third millennium edition, revised and expanded. Springer Monographs in Mathematics. Springer-Verlag, Berlin, 2003. xiv+769 pp. ISBN: 3-540-44085-2   

  
\bibitem{AKA}
 A. Kanamori
{\emph{The higher infinite.
Large cardinals in set theory from their beginnings.}} Second edition. Paperback reprint of the 2003 edition. Springer Monographs in Mathematics. Springer-Verlag, Berlin, 2009.

\bibitem{BAK1}
B. Kastermans
{\emph{The complexity of maximal cofinitary groups.}}
Proc. Amer. Math. Soc. 137 (2009), no. 1, 307-316. 

\bibitem{AK}
 A. Kechris
{\emph{Classical descriptive set theory.}}
Graduate Texts in Mathematics, 156. Springer-Verlag, New York, 1995.


\bibitem{M}
D. Martin
{\emph{Proof of a conjecture of Friedman.}}
Proc. Amer. Math. Soc. 55 (1976), no. 1, 129. 

\bibitem{ADM1}
A. R. D. Mathias
{\emph{Happy families.}}
Ann. Math. Logic 12 (1977), no. 1, 59-111.



\bibitem{AM}
A. Miller
{\emph{Descriptive set theory and forcing.
How to prove theorems about Borel sets the hard way.}} Lecture Notes in Logic, 4. Springer-Verlag, Berlin, 1995.

\bibitem{AM1}
A. Miller
{\emph{Infinite  combinatorics  and  definability.}}
Annals of Pure and Applied Logic, vol. 41 (1989), 179-203.

\bibitem{YM1}
Y. Moschovakis
{\emph{Descriptive set theory.}}
Second edition. Mathematical Surveys and Monographs, 155. American Mathematical Society, Providence, RI, 2009. xiv+502 pp. ISBN: 978-0-8218-4813-5 

\bibitem{RS1}
D. Raghavan, S. Shelah
{\emph{Comparing the closed almost disjointness and dominating numbers.}}
Fund. Math. 217 (2012), no. 1, 73-81. 


\bibitem{JS1}
J. Schilhan
{\emph{Coanalytic ultrafilter bases.}}
In preparation. 

\bibitem{DS1}
D. Schrittesser
{\emph{Compactness of maximal eventually different families.}}
Bull. Lond. Math. Soc. 50 (2018), no. 2, 340-348. 

\bibitem{AT1}
A. T\"ornquist
{\emph{Definability and almost disjoint families.}} 
Adv. Math. 330 (2018), 61-73. 

\bibitem{AT2}
A. T\"ornquist
{\emph{$\Sigma^1_2$ and $\Pi^1_1$ mad families.}}
J. Symbolic Logic 78 (2013), no. 4, 1181-1182. 








\end{thebibliography}
\end{document}